\documentclass{amsart}

\usepackage{amssymb}
\usepackage{amsmath}
\usepackage{amsthm}
\usepackage{amsbsy}
\usepackage{bm}

\theoremstyle{plain}
\newtheorem{theorem}{Theorem}[section]
\newtheorem{lemma}[theorem]{Lemma}

\newtheorem{proposition}[theorem]{Proposition}
\newtheorem{example}[theorem]{Example}

\theoremstyle{definition}
\newtheorem{definition}[theorem]{Definition}

\theoremstyle{remark}

\newcommand{\R}{\mathbb{R}}

\newcommand{\N}{\mathbb{N}}
\newcommand{\D}{\mathbb{D}}

\newcommand{\sm}{\setminus}

\begin{document} 

\title[Metric measure spaces and distance matrices]{Lipschitz correspondence between metric measure spaces and random distance matrices}

\author{Siddhartha Gadgil}

\address{	Department of Mathematics\\
		Indian Institute of Science\\
		Bangalore 560012, India}

\email{gadgil@math.iisc.ernet.in}

\author{Manjunath Krishnapur}
\address{Department of Mathematics\\
		Indian Institute of Science\\
		Bangalore 560012, India}

\email{manju@math.iisc.ernet.in}

\thanks{Partially supported by UGC
(under SAP-DSA Phase IV)}


\date{\today}

\begin{abstract}
Given a metric space with a Borel probability measure, for each integer $N$ we obtain a probability distribution on $N\times N$ distance matrices by considering the distances between pairs of points in a sample consisting of $N$ points chosen indepenedently from the metric space with respect to the given measure. We show that this gives an asymptotically bi-Lipschitz relation between metric measure spaces and the corresponding distance matrices. This is an effective version of a result of Vershik that metric measure spaces are determined by associated distributions on infinite random matrices.
\end{abstract}

\maketitle

\section{Introduction}

Let $(X,d)$ be a metric space and let $\mu$ be a Borel probability measure on $X$ (we shall henceforth refer to $(X,d,\mu)$ as a metric measure space). Consider a sequence $\{x_n\}_{n\in \N}$ of random points in $X$ chosen independently according to the probability measure $\mu$. We obtain a random matrix $D=(d_{ij})=(d(x_i,x_j))$ with rows and columns indexed by the positive integers $\N$. Thus, the triple $(X,d,\mu)$ gives rise to a distribution on random matrices with rows and columns indexed by $\N$ (we shall call these infinite square matrices).

This work is motivated by a result of Vershik~\cite{Ve} that the metric measure space $(X,d,\mu)$ is determined, up to measure preserving isometries, by the corresponding distribution on infinite square matrices. Our goal is to give an effective version of this result for distributions on matrices obtained by choosing a finite (but sufficiently large) collection of points. Our result in fact gives a bi-Lispschitz correspondence.

Namely, for a positive integer $N$, we sample $N$ independent points $x_1$, $x_2$,\dots,$x_N$, from a given compact metric space $X$ according to a given probability distribution $\mu$. This gives a probability distribution on $N\times N$ (symmetric) matrices. We show that there is an asymptotically bi-Lipschitz relation between metric measure spaces and the corresponding distributions on matrices. Here we use a notion of distance on metric measure spaces, which we call the \emph{Gromov-Hausdorff-Prokhorov  distance}, which is a generalisation of the Gromov-Hausdorff distance between metric spaces.

One expects that if two metric measure spaces are close, then the corresponding distributions on matrices are close for all $N$ (not just large $N$). We also prove such a result, but in this case the relation is H\"older with optimal exponent $1/2$. Finally, in Section~\ref{largedev} we introduce a quantity, which we call the relative entropy of metric measure spaces, and conjecture a large deviations result.

We remark that the distributions on matrices are naturally related to metric notions of curvature. In particular, CAT($\kappa$) spaces and, more generally, Wirtinger spaces~\cite{Gr2}, can be viewed as defined in terms of the support of distribution on matrices. On the other hand, the Ricci curvature is related to the measure of cones, which can be related to distance matrices for collections of points with one point fixed and the others chosen at random. It would be interesting to study the continuity of optimal transport with respect to the Gromov-Hausdorff-Prokhorov  distance.

\section{The distance between metric measure spaces}

In this section, we recall the definitions of a metric on probability measures on a given metric space and the Gromov-Hausdorff distance between compact metric spaces. We then introduce a notion of distance between metric measure spaces, which is a combination of these definitions. We  show that our definition indeed gives a metric when we consider metric spaces with measures having full support. Further, we prove that an analogue of Gromov's compactness theorem holds for metric measure spaces.

\subsection{Distance between distributions}
Let $\mu_1$ and $\mu_2$ be Borel probability measures on a given metric space $(Z,d_Z)$. Let $\pi_i:Z\times Z\to Z$, $i=1,2$ be the projection maps. Consider probability measures $\theta$ on $Z\times Z$ so that the marginal distributions satisfy $\pi_{i*}(\theta)=\mu_i$ for $i=1,2$. For such a measure $\theta$, we define $\Delta(\theta)=\Delta(\theta;d_Z)$ by
$$\Delta(\theta)=\inf\{r>0: \theta(\{(z_1,z_2)\in Z\times Z: d_Z(z_1,z_2)\leq r\})\geq 1-r\}.$$ 

We define the L\'evy-Prokhorov distance between $\mu_1$ and $\mu_2$ to be the infimum, 
$$d_P(\mu_1,\mu_2)=\inf\{\Delta(\theta): \text{$\theta$ measure on $Z\times Z$}, \ \pi_{i*}(\theta)=\mu_i,i=1,2\}.$$

We shall sometimes denote this as $d_P(\mu_1,\mu_2;Z)$ or $d_P(\mu_1,\mu_2;d_Z)$ to clarify the underlying metric space.

Note that the distance $d_P$ has another equivalent formulation. Namely, we consider random variables $X_i(\omega)\in Z$, $i=1,2$, on a sample space $\Omega$ with probability measure $P$, so that the distributions of $X_i(\omega)$ is $\mu_i$ for $i=1,2$. For such a pair $(X_1,X_2)$ of random variables, we consider $$\Delta(X_1,X_2)=\inf\{r>0: P(\{\omega\in \Omega: d_Z(X_1(\omega),X_2(\omega))\leq r\})\geq 1-r\}.$$ 
Then $\Delta(\theta)$ is the infimum of $\Delta(X_1,X_2)$ over all pairs $(X_1,X_2)$ of random variables so that the marginal distribution of $X_i$ is $\mu_i$ for $i=1,2$ (we can see  by considering the pushforward of the measure $P$ on $\Omega$ to $Z\times Z$ using the map $\omega\mapsto (X_1(\omega),X_2(\omega))$).

\subsection{Gromov-Hausdorff distance}

Let $(X_1,d_1)$ and $(X_2,d_2)$ be compact metric spaces.  Consider pairs of isometric embeddings $\iota_i:X_1\to Z$, $i=1,2$ of $X_1$ and $X_2$ into a metric space $Z$. For such embeddings, we can consider the Hausdorff distance $d_H$ between $\iota_1(X_1)$ and $\iota_2(X_2)$  (as subsets of  $Z$).

The Gromov-Hausdorff distance~\cite{Gr1} between $X_1$ and $X_2$ is defined as the infimum of such Hausdorff distances, i.e.,
$$d_{GH}(X_1,X_2)=\inf\{d_H(\iota_1(X_1),\iota_2(X_2)): \text{$\iota_i:X_i\to Z$ isometric embeddings}\}$$
where $Z$ in the infimum varies over all (compact) metric spaces. 

\subsection{Distance between metric measure spaces} We now define the \emph{Gromov-Hausdorff-Prokhorov } distance between a pair of metric measure spaces $(X_i,d_i,\mu_i)$, $i=1,2$, with the underlying metric space $X_i$ assumed to be compact. Consider isometric embeddings $\iota_i:X_i\to Z$, $i=1,2$, of the spaces $X_i$ into a metric space $Z$. These give rise to pushforward probability measures $\iota_{i*}(\mu_i)$ on $Z$. The distance between the metric measure spaces is the infimum of the distance between the pushforward probability measures over all isometric embeddings, i.e.,
$$d_{GHP}(X_1,X_2)=\inf\{d_P(\iota_1(\mu_1),\iota_2(\mu_2)): \text{$\iota_{i*}:X_i\to Z$ isometric embeddings}\}$$
where $Z$ in the infimum varies over all (compact) metric spaces.

We can identify the spaces $X_i$ with their images in $Z$. Further, we can assume that, under these identifications, $Z=X_1\cup X_2$. We shall often make such identifications and identify the measures $\mu_i$ with the corresponding pushforward measures on $Z$. Further, we often suppress the  measures from the notation if they are clear from the context.

It is clear that the distance is symmetric. We shall prove the triangle inequality, showing that we get a pseudo-metric. We also show an appropriate positivity result, showing in particular that we get a genuine metric on metric measure spaces for which the measure has full support. 

We remark that the definition of the Gromov-Hausdorff-Prokhorov  distance works even for pseudo-metric spaces, i.e., where we allow $d(x,y)=0$ even if $x\neq y$. All our results also hold in this case.

\subsection{The triangle inequality}

Let $(X_1,\mu_1)$, $(X_2,\mu_2)$ and $(Y,\nu)$ be metric, measure spaces. 

\begin{proposition}\label{triang}
We have
$$d_{GHP}(X_1,X_2)\leq d_{GHP}(X_1,Y)+ d_{GHP}(X_2,Y).$$
\end{proposition}
\begin{proof}
Let $\epsilon>0$ be arbitrary. By definition, for $i=1,2$, we can find spaces $Z_i=X_i\cup Y$ so that the distance between (the pushforwards of) the measures $\mu_i$ and $\nu$ is at most $d_{GHP}(X_i,Y)+\epsilon$. 

Now, let $Z$ be the metric space obtained from  $Z_1\coprod Z_2$ by identifying the isometric copies of $Y$, and with distance the maximal metric whose restriction to each $Z_i$ is the given metric. More concretely, the metric $d_Z$ on $Z$ is given by
\begin{enumerate}
\item If both $x_1$ and $x_2$ lie in some $Z_i$, $d_Z(x_1,x_2)=d_{Z_i}(x_1,x_2)$ 
\item If $x_1\in Z_1$ and $x_2\in Z_2$, 
$d_Z(x_1,x_2)=\inf\{d_{Z_1}(x_1,y)+d_{Z_2}(y,x_2): y\in Y\}.$
\end{enumerate}
It is easy to see that $d_Z$ defined as above gives a metric whose restriction to each $Z_i$ is the given metric on $Z_i$. The triangle inequality implies that this is indeed the maximal such metric. 

The measures $\mu_i$ and $\nu$ push forward to give measures on $Z$, with the the distance between the pushforwards in $Z$ of the measures $\mu_i$ and $\nu$  at most $d_{GHP}(X_i,Y)+\epsilon$ for $i=1,2$. By the triangle inequality for $d_P$, it follows that
$$d_{GHP}(X_1,X_2)\leq d_P(\mu_1,\mu_2)\leq d_{GHP}(X_1,Y)+ d_{GHP}(X_2,Y)+2\epsilon.$$

As $\epsilon>0$ was arbitrary, the result follows.
\end{proof}

\subsection{Positivity of the distance function}

As the definition of $d_{GHP}$ ignores, for example, isolated points that have measure $0$ (which is to be expected for stochastic concepts), we do not expect  $d_{GHP}(X,Y)=0$ to imply that $X=Y$, but only that this is true up to ignoring an appropriate class of sets with measure $0$. We now prove such a result.

\begin{theorem}\label{posit}
For two metric, measure spaces $(X,\mu)$ and $(Y,\nu)$, $d_{GHP}(X,Y)=0$ if and only if there are open sets $U\subset X$ and $V\subset Y$ of measure $0$ so that there is a measure preserving isometry between $X\sm U$ and $Y\sm V$.
\end{theorem}
\begin{proof}
Given a measure preserving isometry between $X\sm U$ and $Y\sm V$, with $U$ and $V$ open sets with measure zero, we take $Z$ to be the space obtained from $X\coprod Y$ by identifying $X\sm U$ and $Y\sm V$ using the given measure preserving isometry, with the metric on $Z$ the maximal metric whose restrictions to $X$ and $Y$ are the given metrics (analogous to the one constructed in Proposition~\ref{triang}). Then the pushforward measures on $Z$ under the inclusion maps are equal. It follows that $d_{GHP}(X,Y)=0$.

Conversely, by hypothesis, there is a sequence of metric spaces $Z_i=X\cup Y$ so that the distance between the pushforwards in $Z_i$ of the measures $\mu$ and $\nu$ converges to $0$. We shall first construct a limit of these spaces.

It is easy to see that there is a uniform bound on the diameter of the spaces $Z_i$. Further, as $Z_i=X\cup Y$, given $\epsilon>0$ there is an integer $N=N(\epsilon)$, independent of $i$, so that there is a finite set $F_i\subset Z_i$ with cardinality at most $N$ so that the each point in $Z_i$ has distance at most $\epsilon$ from $F_i$. Hence, by Gromov's compactness theorem for metric spaces, on passing to a subsequence, $Z_i$ converges in the Gromov-Hausdorff metric to a compact metric space $Z$.

By an equivalent definition of the Gromov-Hausdorff distance, we have a sequence $\epsilon_i\to 0$ and maps $\varphi_i:Z_i\to Z$ so that 
$$\vert d(\varphi_i(p),\varphi_i(q))-d(p,q)\vert<\epsilon_i,\ \forall p,q\in Z_i .$$

Hence, on composing with the maps from $X$ and $Y$ to $Z_i$, we get maps $f_i:X\to Z$ and $g_i:Y\to Z$ that satisfy the analogous conditions
$$\vert d(f_i(p),f_i(q))-d(p,q)\vert<\epsilon_i,\ \forall p,q\in X, $$
and
$$\vert d(g_i(p),g_i(q))-d(p,q)\vert<\epsilon_i,\ \forall p,q\in Y.$$
 As in the proof of the Arzela-Ascoli theorem, we can pass to a subsequence to obtain limiting maps $f:X\to Z$ and $g:Y\to Z$ that are isometric embeddings.

Further, we have measures $\theta_i$ on $Z_i\times Z_i$ with marginals the pushforwards of the measures $\mu$ and $\nu$ to $Z_i$, so that $\Delta(\theta_i)\to 0$. On passing to a subsequence, the pushforwards to $Z\times Z$ of the measures $\theta_i$ converge to a measure $\theta$ on $Z$, with marginals of $\theta$ the pushforwards of the measures $\mu$ and $\nu$ and with $\Delta(\theta)=0$. It follows that the pushforward measures coincide.

In particular, identifying $X$ and $Y$ with their images, the measures are supported on $X\cap Y$ ($X\cap Y$ is defined by viewing $X$ and $Y$  as subsets of $Z$). Hence, if $U=X\sm (X\cap Y)$ and $V=Y\sm (X\cap Y)$, $U$ and $V$ are open sets of zero measure in $X$ and $Y$, respectively (as $X\cap Y$ is compact, hence closed in $X$ and $Y$), and there is a measure preserving isometry between $X\sm U$ and $Y\sm V$. 

\end{proof}

In particular if we restrict to metric measure spaces with the measure having full support, $d_{GHP}$ is a genuine metric.

\subsection{A compactness theorem}

We shall show that an analogue of Gromov's compactness theorem holds for metric measure spaces.

\begin{theorem}\label{cpct}
Let $(X_n,d_n,\mu_n)$ be a sequence of compact metric measure spaces so that 
\begin{enumerate}
\item there is a uniform bound $D>0$ on the diameter of the spaces $X_n$.
\item for each $\epsilon>0$ there is an integer $N(\epsilon)$ so that, for each $n\in\N$, $X_n$ contains an $\epsilon$-net with cardinality at most $N(\epsilon)$.
\end{enumerate}
Then there is a subsequence $X_{n_i}$ that converges in the Gromov-Hausdorff-Prokhorov  metric.
\end{theorem}
\begin{proof}
By Gromov's compactness theorem for metric spaces, on passing to a subsequence, we can ensure that the metric spaces $(X_n,d_n)$ converge to a metric space $(Z,d_Z)$. By an equivalent formulation of the Gromov-Hausdorff distance, we can choose the subsequence so that there are functions $f_n:X_n\to Z$ so that for all $n$,
\begin{equation}\label{almstisom}
\vert d_Z(f_n(x),f_n(y))-d_n(x,y)\vert<\frac{1}{n}.
\end{equation}

Consider the pushforward measures $\eta_n=f_{n*}(\mu_n)$ on $Z$. We first bound the distance between $(X_n,d_n,\mu_n)$ and $(Z,d_Z,\eta_n)$. Namely, consider the maximal metric $d_W$ on $W=X_n\coprod Z$ so that 
\begin{enumerate}
\item if $x_1,x_2\in X_n$, $d_W(x_1,x_2)\leq d_n(x_1,x_2)$.
\item if $x_1,x_2\in Z$, $d_W(x_1,x_2)\leq d_Z(x_1,x_2)$.
\item if $x\in X_n$, $d(x,f_n(x))\leq 1/n$.
\end{enumerate}
Concretely, we consider $d'_W:W\times W\to\R$ to be the symmetric function defined by:
\begin{enumerate}
\item if $x_1,x_2\in X_n$, $d'_W(x_1,x_2)=d_n(x_1,x_2)$.
\item if $x_1,x_2\in Z$, $d'_W(x_1,x_2)= d_Z(x_1,x_2)$.
\item if $x_1\in X_n$ and $x_2\in Z$, then \\
$d'_W(x_1,x_2)=\inf\{d(x_1,x)+\frac{1}{n}+d(f(x),x_2): x\in X_n\}. $
\end{enumerate}
The function $d'_W$ satisfies the triangle inequality and hence gives a metric on $W$. Further, $d'_W$ satisfies the conditions required by $d_W$. By the triangle inequality we can see that any metric satisfying the conditions required for $d_W$ is bounded above by $d'_W$, so $d_W=d'_W$ by maximality.

In particular the inclusion maps from $X_n$ and $Z$ are isometric embeddings in $Z_n$ (as they are for the metric $d'_W$ by construction). Let $\nu_n$ be the measure on the product $Z_n\times Z_n$ obtaining by pushing forward the measure $\mu_n$ using the map $x\mapsto (x,f_n(x))$. Then the marginals of $\nu_n$ are $\mu_n$ and $\eta_n$ and $\Delta(\nu_n)\leq 1/n$. This shows that $d_{GHP}((X_n,d_n,\mu_n), (Z,d_Z,\eta_n))\leq 1/n$.

As the measures $\eta_n$ are all Borel probability measures, on passing to a subsequence these converge to a measure $\eta$. We can thus assume that $d_P(\eta_n,\eta)<1/n$, and hence $d_{GHP}((Z,d_Z,\eta_n), (Z,d_Z,\eta))\leq 1/n$. By the triangle inequality, it follows that
$d_{GHP}((X_n,d_n,\mu_n), (Z,d_Z,\eta))\leq \frac{2}{n}$.

\end{proof}

\section{Distance between square matrices}
To make uniform statements, it will be convenient to introduce a modified distance function on the space $M(N)$ of $N\times N$ symmetric square matrices by allowing some rows and corresponding columns to be excluded. Namely, for matrices $A=(a_{ij})$ and $B=(b_{ij})$, we let
$$d_M(A,B)=\inf\{\rho>0: \exists \lambda\subset\{1,2,\dots,N\}, \vert\lambda\vert<N\rho,\ d(a_{ij},b_{ij})<\rho\ \forall i,j\notin\lambda\}$$ 

Note that this is a pseudo-metric, but we can use pseudo-metrics in place of metrics in most constructions. Further, note that for any fixed $N$ this coincides with the supremum metric for a pair of matrices $A,B$ with $d_M(A,B)<1/N$. 

The permutation group $S_N$ on $N$ letters acts on $M(N)$ by simultaneous permutations of rows and columns. We have a corresponding distance $d_\pi$, with matrices close in $d_\pi$ if they are close in the previous sense up to permutation. More precisely,
$$d_\pi(A,B)=\min\{d_M(A,\pi B): \pi\in S_N\}.$$

These induce distances on the space of distributions on $M(N)$. We observe that these induced distances coincide for distributions that are invariant under the symmetric group. Let $\mu_1$ and $\mu_2$ be distributions on $M(N)$ that are invariant under the action of $S_N$. We remark that the following proposition holds (with the same proof) for more general group actions.

\begin{proposition}\label{gpaction}
For measures $\mu_1$ and $\mu_2$ invariant under the action of $S_N$, $$d_P(\mu_1,\mu_2;d_M)=d_P(\mu_1,\mu_2,d_\pi).$$
\end{proposition}
\begin{proof}
As $d_M(\cdot,\cdot)\geq d_\pi(\cdot,\cdot)$, it follows that $d_P(\mu_1,\mu_2;d_M)\geq d_P(\mu_1,\mu_2;d_\pi)$.  

The converse follows from the following lemma.
\begin{lemma}
Given measures $\mu_1$ and $\mu_2$ on $M(N)$ that are invariant under the action of $S(N)$ and a measure $\nu$ on $M(N)\times M(N)$ with marginals $\mu_i$, there is a measure $\nu'$ with marginals $\mu_i$ so that $\Delta(\nu';d_\pi)=\Delta(\nu;d_M)$. 
\end{lemma}
\begin{proof}
Given matrices $A$ and $B$, there is an element $\sigma\in S_N$ so that we have $d_M(A,\sigma B)=d_\pi(A,B)$. Let $\psi:M(N)\times M(N)\to M(N)\times M(N)$ be a measurable function that associates to $(A,B)$ a pair $(A,\sigma B)$ so that $d(A,\sigma B)=d_\pi(A,B)$. We define $\nu''=\psi_*(\nu)$, i.e., the pushforward of $\nu$ under the map $\psi$, and let $\nu'$ be obtained from $\nu''$ by averaging with respect to the diagonal action of $S_N$. 

By construction, if $\mu_i'$ and $\mu_i''$ are the marginals of $\nu'$ and $\nu''$, then for an orbit $S_N A$ of a matrix $A$ we have $\mu_i(S_N A)=\mu''_i(S_N A)=\mu'_i(S_N A)$ (as the constructions of $\nu''$ from $\nu$ and $\nu'$ from $\nu''$ leave the marginal measure of an orbit unchanged). Further, by construction the marginals $\mu'_i$ are invariant under the action of $S(N)$; by hypothesis, the measures $\mu_i$ are also invariant under this action. It follows that $\mu_i'=\mu_i$, so $\mu_i$ are the marginals of $\nu'$. By construction $\Delta(\nu';d_\pi)=\Delta(\nu;d_M)$. 
\end{proof}

Now, let $\nu$ be a measure on $M(N)\times M(N)$ with marginals $\mu_i$ satisfying  $$\Delta(\nu;d_\pi)<d_P(\mu_1,\mu_2;d_\pi)+\epsilon.$$ 
By the above lemma there is a measure $\nu'$ with marginals $\mu_i$ so that we have $\Delta(\nu';d_M)=\Delta(\nu;d_\pi)$. This implies that 

$$d_P(\mu_1,\mu_2;d_M)<d_P(\mu_1,\mu_2;d_\pi)+\epsilon.$$ 
As $\epsilon>0$ is arbitrary, the claim follows.
\end{proof}
 
\section{Finite metric spaces and Distance matrices}

A \emph{distance matrix} is a real symmetric matrix $A=(a_{ij})$ with non-negative entries and zeroes on the diagonal which satisfies the triangle inequality $$a_{ij}\leq a_{ik}+a_{kj}$$ for all $i$, $j$ and $k$. Let $\mathbb{D}(N)$ be the subset of $M(N)$ consisting of distance matrices. We shall consider the space $\D(N)$ with metrics obtained by restricting $d_\pi$ and $d_M$. 

Let $X=\{x_1,\dots,x_n\}$ be a finite set with a pseudo-metric $d$. We can associate to $X$ a distance matrix $A=(a_{ij})$, with $a_{ij}=d(x_i,x_j)$. 

Conversely, there is a natural map $\Theta$ from $\D(N)$ to pseudo-metric spaces  carrying  Borel probability measures. Namely, we associate to $A=(a_{ij})$ the space $\Theta(A)$ with points $x_i$, $1\leq i\leq N$, corresponding to the rows (equivalently the columns) of $A$ and the distance given by $d(x_i,x_j)=a_{ij}$. The measure of each singleton set $\{x_i\}$ is defined to be $1/N$. This construction gives a pseudo-metric and measure. We remark that we can get a metric space by identifying points whose distance is zero and pushing forward the measure under the corresponding quotient map. 

We show that $\Theta$ is a bi-Lipschitz map, with Lipschitz constant $2$.

\begin{theorem}\label{finspc}
For $A,B\in\D(N)$, we have 
$$d_{GHP}(\Theta(A),\Theta(B))\leq d_\pi(A,B)\leq 2d_{GHP}(\Theta(A),\Theta(B)).$$
\end{theorem}
\begin{proof}
Suppose $A=(a_{ij})$ and $B=(b_{ij})$ are distance matrices in $D(N)$.  The spaces $X=\Theta(A)$ and $Y=\Theta(B)$ have $N$ points, which we denote $x_1$, $x_2$, \dots, $x_N$ and $y_1$, $y_2$,\dots, $y_N$ with $d_X(x_i,x_j)=a_{ij}$ and $d_Y(y_i,y_j)=b_{ij}$. The measures on the spaces $X$ and $Y$ assign a measure of $1/N$ to each point.

Suppose $d_\pi(A,B)<\epsilon$. By permuting the rows and columns of $A$ and $B$, we can assume that $\vert a_{ij}-b_{ij}\vert<\epsilon$ if $1\leq i,j<N(1-\epsilon)$. Let $Z=X\coprod Y$ with metric $d_Z$ the maximal metric so that 
\begin{enumerate}
\item $d_Z(x_i,x_j)\leq d_X(x_i,x_j)$ for $1\leq i,j\leq N$.
\item $d_Z(y_i,y_j)\leq d_Y(y_i,y_j)$ for $1\leq i,j\leq N$.
\item $d_Z(x_i,y_i)\leq \epsilon$ for $1\leq i,j\leq N(1-\epsilon)$.
\end{enumerate}
As in the analogous construction in Theorem~\ref{cpct}, we see that $X$ and $Y$ isometrically embed in $Z$. Further, by considering the measure on $Z\times Z$ which assigns the weight $1/N$ to points of the form $(x_i,y_i)$, $1\leq i\leq N$, and zero to all other points, it follows that $d_{GHP}(X,Y)<\epsilon$.

Conversely, suppose, for some $\epsilon>0$,  $d_{GHP}(X,Y)<\epsilon$. It follows that there is a space $Z$ with subspaces that can be identified with $X$ and $Y$ so that $Z=X\cup Y$ and the pushforward measures $\mu_X$ and $\mu_Y$ satisfy $d_P(\mu_X,\mu_Y)<\epsilon$. Let $\nu$ be a measure on $Z\times Z$ with marginals $\mu_X$ and $\mu_Y$ so that $\Delta(\nu)<\epsilon$. We shall consider a bijection between the points of $X$ and $Y$, which will give the necessary permutation of the rows and columns of $A$ and $B$. 

We now prove the existence of the desired bijection. A second proof  is given later.

Let $X_0=\{x\in X:\exists y\in Y\text{\ such\ that\ } d(x,y)<\epsilon\}$. Then there exist functions $f:X_0\to Y$ so that $d(x, f(x))<\epsilon$ for all $x\in X_0$. Choose such a function so that the cardinality (equivalently the measure) of $f(X_0)$ is maximised. 

\begin{lemma}
For $f$ maximal as above, $\mu_Y(Y\sm f(X_0))<N\epsilon$. 
\end{lemma}
\begin{proof}
We define subsets $X_1\subset X$ and $Y_1\subset Y$. Namely, for $n\geq 0$, we define a \emph{permissible chain} to be a sequence of elements of the form  $y_0$, $x_1$, $y_1$, \dots, $x_n$, $y_n$ so that
\begin{enumerate}
\item $y_0\in Y\sm f(X_0)$,
\item\label{close} $d(x_i,y_{i-1})<\epsilon$, $1\leq i\leq n$,
\item\label{image} $y_i=f(x_i)$, $1\leq i\leq n$.
\end{enumerate} 

We also allow $n=0$ for permissible chains. Observe that condition~(\ref{close}) implies that $x_i\in X_0$, so $f(x_i)$ is defined and condition~(\ref{image}) makes sense.

Let $X_1\subset X$ and $Y_1\subset Y$ be the sets of elements of the form $x_n$ and $y_n$, respectively, for permissible chains $y_0$, $x_1$, $y_1$, \dots, $x_n$, $y_n$. Note that in particular $Y\sm f(X_0)\subset Y_1$ (by considering chains with $n=0$).
 
Note that if $y\in Y_1$ and $x\in X$ satisfies $d(x,y)<\epsilon$, then $x\in X_1$. In particular, as $Y\sm f(X_0)\subset Y_1$, we have
\begin{equation}\label{seteq}
Y\sm f(X_0)=Y_1\sm f(X_0)=Y_1\sm f(X_1).
\end{equation}

We claim that $f$ is injective on $X_1$. For, if $f(p)=f(q)$ with $p\in X_1$, then there is a permissible chain $y_0$, $x_1$, $y_1$, \dots, $x_n$, with $p=x_n$. Without loss of generality, we can assume that $q\neq x_i$ for $1\leq i\leq n$. We define a function $g:X_0\to Y$ by $g(x_i)=y_{i-1}$ for $1\leq i\leq n$ and $g(x)=f(x)$ if $x\notin\{x_1,\dots,x_n\}$. Observe that we have $d(x,g(x))<\epsilon$ for all $x\in X_0$ and $g(X_1)=f(X_1)\cup \{y_0\}$ (as $f(p)=f(q)=g(q)$ is in the image of $g$). This contradicts the maximality of the image of $f$. Thus, $f$ must be injective on $X_1$.

As $f:X_1\to Y_1$ is injective and hence measure preserving, $\mu_Y(f(X_1))=\mu_X(X_1),$
hence
\begin{equation}\label{Y1split}
\mu_Y(Y_1)=\mu_X(X_1)+\mu_Y(Y_1\sm f(X_1)).
\end{equation}
Further, we have,
\begin{equation}\label{Y1nu}
\mu_Y(Y_1)=\nu(X\times Y_1)=\nu(X_1\times Y_1)+\nu((X\sm X_1)\times Y_1).
\end{equation}
For each pair $(x,y)\in (X\sm X_1)\times Y_1$, $d(x,y)\geq \epsilon$. It follows that $\nu(X\sm X_1)\times Y_1<\epsilon$. Further, $\nu(X_1\times Y_1)\leq\mu_X(X_1)$. Substituting these estimates in Equation~\ref{Y1nu} and using Equation~\ref{Y1split}, we obtain
\begin{equation}
\mu_X(X_1)+\mu_Y(Y_1\sm f(X_1))=\mu_Y(Y_1)\leq \mu_X(X_1)+\epsilon.
\end{equation}
Finally, using Equation~\ref{seteq}, we obtain  
\begin{equation}
\mu(Y\sm f(X_0))=\mu(Y_1\sm f(X_1))<\epsilon.
\end{equation}
as desired.
\end{proof}

It follows that $f:X_0\to Y$ has image with complement having at most $N\epsilon$ points. We can replace $X_0$ by a subset to make $f$ injective without changing its image. Extending this arbitrarily gives a bijection between points of $X$ and $Y$ so that the distances between $x$ and $f(x)$ is at most $\epsilon$ for at least $N(1-\epsilon)$ points in $X$. By applying permutations, we can assume that $d(X_i,Y_i)<\epsilon$ for $1\leq i\leq N(1-\epsilon)$. By the triangle inequality we deduce that $\vert a_{ij}-b_{ij}\vert <2\epsilon$ for $1\leq i\leq N(1-\epsilon)$. 

Thus, $d_\pi(A,B)<2\epsilon$. As this holds whenever $d_{GHP}(X,Y)<\epsilon$, we see that
$$d_\pi(A,B)\leq 2d_{GHP}(\Theta(A),\Theta(B)),$$
as desired.
\end{proof} 

Now we sketch the promised second proof.
\begin{proof}[Alternate proof of the second inequality in Theorem~\ref{finspc}] 
 Suppose $d_{GHP}(X,Y)<\epsilon$ and as before consider a probability measure $\nu$ on $Z\times Z$ with $\Delta(\nu)<\epsilon$ and having marginals $\mu_{X}$ and $\mu_{Y}$.  Let $\nu_{i,j}=\nu\{(x_{i},y_{j})\}$.

The images of  $X$ and $Y$ in $Z$ may be assumed to have $N$ distinct points each, so that each of $\mu_X$ and $\mu_{Y}$ give mass $1/N$ to $N$ distinct points. Then the matrix $\left(N\nu_{i,j}\right)_{i,j\le N}$ is a doubly stochastic matrix. By Birkoff's theorem, it can be written as a convex combination  $\sum c_{\sigma}\sigma$ of permutation matrices.  Hence,
$$
\epsilon>\Delta(\nu) = \sum_{i,j}\nu_{i,j}{\bf 1}_{d(x_{i},y_{j})>\epsilon} = \sum_{\sigma}c_{\sigma} \frac{1}{N}\sum_{i,j}\sigma_{i,j}{\bf 1}_{d(x_{i},y_{j})>\epsilon}.
$$
As $c_{\sigma}$ are non-negative and sum to one, it follows that there exists a permutation $\sigma$ such that 
$$
\epsilon >\frac{1}{N}\sum_{i,j}\sigma_{i,j}{\bf 1}_{d(x_{i},y_{j})>\epsilon} = \frac{1}{N}\sum_{i=1}^{N}{\bf 1}_{d(x_{i},y_{\sigma(i)})>\epsilon}.
 $$ 
Omit all $i$ such that $d(x_{i},y_{\sigma(i)})<\epsilon$. If $i,j$ are among the remaining $N(1-\epsilon)$ indices, then  $|a_{i,j}-b_{i,j}|\le 2\epsilon$ by the  triangle inequality. Thus, $d_{\pi}(A,B)\le 2\epsilon$ as required to show.
\end{proof}

\vspace{2mm}
We remark that it is not true that if two finite metric spaces are close in the Gromov-Hausdorff distance, then their distance matrices are close. For example, for $\epsilon>0$ small consider the subspaces of $\R$ given by
$$X=\{-\epsilon,0,\epsilon,1\},$$
and $$Y=\{0,\epsilon,1,1+\epsilon).$$
As the Hausdorff distance between these, as subsets of $\R$, is $\epsilon$, it follows that $d_{GH}(X,Y)\leq\epsilon$. On the other hand, for $\epsilon$ small, the distance between the corresponding distance matrices is at least $1/4$, as if we exclude less than $1/4$th of the rows, i.e., no row, then the distance between some pair of corresponding entries of the two matrices is greater than $1/2$.

It would be interesting to understand the relation between Gromov-Hausdorff and Gromov-Hausdorff-Prokhorov  convergence for Riemannian manifolds with the normalised volume measure, especially in the case of \emph{collapsing} with bounded sectional curvatures. 

\section{Limits of Samples}

Let $(X,d,\mu)$ be a compact, metric measure space. In this section we show that the empirical space from a sufficiently large sample is close to $X$ in the Gromov-Hausdorff-Prokhorov  metric. 

\begin{definition}\label{empspc}
Suppose $x_1$, $x_2$, \dots, $x_N$ are $N$ points in $X$, the \emph{empirical space} is the pseudo-metric measure space consisting of these points with distance induced from $X$ and measure associating equal weight to each sample point.
\end{definition}

Let $\epsilon>0$ be fixed. As $X$ is compact, there are finitely many points $a_1$,\dots,$a_n$ in $X$ so that each $x\in X$ satisfies $d(x,a_i)<\epsilon$ for some $a_i$. Further, we can partition $X$ into disjoint measurable sets $A_i$, $1\leq i\leq n$, so that, if $x\in A_i$, then $d(x,a_i)<\epsilon$. Let $\psi:X\to \{a_1,\dots,a_n\}$ be the function that maps each $x\in A_i$ to $a_i$.

Consider the metric measure space $\widehat{X}$ consisting of the points $\{a_i\}$ with distances induced from $X$, and with measure $\widehat{\mu}$ given by assigning weight $\mu(A_i)$ to the point $a_i$. Note that $\widehat{\mu}$ can also be regarded as a measure on $X$.

\begin{lemma}\label{hat}
We have $d_{GHP}(\widehat{X},X)\leq \epsilon$.
\end{lemma}
\begin{proof}
Let $Z=X$ and the maps from $\widehat{X}$ and $X$ to $Z$ be the inclusion and the identity. We estimate the distance between the pushforward measures, which we identify with $\widehat{\mu}$ and $\mu$.

Namely, we consider the measure $\nu$ on $X\times X$ supported on the union of sets of the form $\{a_i\}\times A_i$ and so that, for $S\subset A_i$,  $\nu(\{a_i\}\times S)=\mu(S)$. Then this has marginals $\widehat{\mu}$ and $\mu$ and satisfies $\Delta(\nu)\leq \epsilon$.
\end{proof}

Now consider a sample of $N$ points from $X$ chosen independently according to the measure $\mu$. Let $X_N$ with measure $\mu_N$ be the corresponding empirical space. Let $\widehat{X}_N$ be the metric measure space with underlying space $\widehat{X}$ and measure $\psi_*(\mu_N)$.

\begin{lemma}\label{NN}
We have $d_{GHP}(\widehat{X}_N,X_N)\leq \epsilon$.
\end{lemma}
\begin{proof}
Let $Z=X$ and the maps from the spaces to $X$ be the inclusion maps. We estimate the distance between the pushforward measures, which are the measures $\psi_*(\mu_N)$ and $\mu_N$ regarded as measures on $X$. Namely, we define a measure $\nu$ on $X\times X$ with support points of the form $(a_i,b)$, $b\in A_i\cap X_N$, with such a point having weight $\mu_N(b)$. This then has the appropriate marginals (as each point $b$ is in a unique set $A_i$, $\pi_{2*}(\nu)=\mu_N$) and satisfies $\Delta(\nu)\leq\epsilon$.
\end{proof}

We are thus reduced to comparing two distributions on a finite metric space, namely $\widehat{X}$. Let $S=\{a_1,\dots,a_n\}$ be a finite metric space. Consider two probability mass functions $\{p_i\}$ and $\{q_i\}$ on $S$.

\begin{lemma}\label{samespc}
$d_{GHP}((S,\{p_i\}),(S,\{q_i\}))\leq 1-\sum\limits_{i=1}^n \min(p_i,q_i)\leq \sum\limits_{i=1}^n \vert p_i-q_i\vert$.
\end{lemma}
\begin{proof}
We can choose a measure on $S\times S$ with marginals $(p_i)$ and $(q_i)$ so that the point $(i,i)$ has weight at least $\min(p_i,q_i)$. This gives the first inequality. The second inequality follows from the first as $\sum_i p_i=1$.
\end{proof}

Finally, we observe that samples are close to the given distribution on $X$.

\begin{lemma}\label{sampconv}
For $N$ sufficiently large, the probability that $d_{GHP}(X_N,X)>3\epsilon$ is less than $\epsilon$. 
\end{lemma}
\begin{proof}
By the law of large numbers, if $p_i=\mu(A_i)$, then  given a sample $X_N$ with $N$ points, $q_i=\mu_N(A_i)$ converges to $p_i$ as $N\to\infty$. Hence, for $N$ sufficiently large, $P(\vert q_i-p_i\vert>\epsilon/n)<\epsilon/n$. By Lemma~\ref{samespc}, 
\begin{equation}\label{hathat}
P(d_{GHP}(\widehat{X}_N,\widehat{X})>\epsilon)<\epsilon
\end{equation}
Lemmas~\ref{hat} and ~\ref{NN}, Equation~\ref{hathat} and the triangle inequality show that
$$P(d_{GHP}(X_N,X)>3\epsilon)<\epsilon$$
as claimed.
\end{proof}

\section{The Asymptotically Lipschitz correspondence}

We can now show that the correspondence between metric measure spaces and distributions on distance matrices is asymptotically bi-Lipschitz. Let $(X,d,\mu)$ be a metric measure space. For $N\in\N$, let $x_1$, $x_2$,\dots, $x_N$ be a sequence of points in $X$ sampled independently according to $\mu$. We associate to the random sample $x_1$, $x_2$,\dots, $x_N$  the distance matrix $M_N^X$. We denote the distribution of $M_N^X$ by $\theta_N^X$. We thus have  
a sequence of distributions $\theta_N^X$ on the metric spaces $\D(N)$ associated to $X$.

Consider a pair of metric measure spaces $X$ and $Y$ and the corresponding sequence of probability distributions $\theta_N^X$ and $\theta_N^Y$. In all the statements in this and the next section, Proposition~\ref{gpaction} shows that we can replace $d_P(\cdot,\cdot;d_\pi)$ by $d_P(\cdot,\cdot;d_M)$.

\begin{theorem}
For metric measure spaces $X$ and $Y$, we have
\begin{enumerate}
\item $\limsup d_P(\theta_N^X,\theta_N^Y;d_\pi)\leq 2d_{GHP}(X,Y)$.
\item If  $\liminf d_P(\theta_N^X,\theta_N^Y;d_\pi)<\epsilon<1$, then $d_{GHP}(X,Y)<\epsilon$.
\end{enumerate}
\end{theorem}
\begin{proof}
Suppose $\epsilon>0$ is arbitrary. For $N>0$ an integer, let $X_N$ and $Y_N$ be empirical spaces determined by choosing $N$  points in each of $X$ and $Y$,  such that all the $2N$ points are chosen independently. Let $M^X_N$ and $M^Y_N$ be the corresponding random distance matrices. Note that these have distributions $\theta_N^X$ and $\theta_N^Y$, respectively.

If $N$ is sufficiently large, by Lemma~\ref{sampconv}, $d_{GHP}(X_N,X)<3\epsilon$ with probability at least $1-\epsilon$ and $d_{GHP}(Y_N,Y)<3\epsilon$ with probability at least $1-\epsilon$. Hence $d_{GHP}(X_N,Y_N)<d_{GHP}(X,Y)+6\epsilon$ with probability at least $1-2\epsilon$. 
By Theorem~\ref{finspc}, $d_\pi(M_N^X,M_N^Y)\leq 2d_{GHP}(X_N,Y_N)$.
It follows that for $N$ sufficiently large
$$P(d_\pi(M^X_N,M^Y_N)<2d_{GHP}(X,Y)+12\epsilon)>1-2\epsilon.$$

As $\epsilon>0$ is arbitrary and $M_N^X$ and $M_N^Y$ have distributions $\theta_N^X$ and $\theta_N^Y$, respectively, it follows that  
$$\limsup_{N\to\infty} d_P(\theta_N^X,\theta_N^Y;d_\pi)\leq 2d_{GHP}(X,Y).$$  

On the other hand, suppose $\liminf d_P(\theta_N^X,\theta_N^Y;d_M)<\epsilon<1$. Let $X_N$, $Y_N$, $M_N^X$ and $M_N^Y$ be as before. Then for infinitely many $N$, $P(d_\pi(M_N^X,M_N^Y)<\epsilon)>1-\epsilon$. By Theorem~\ref{finspc}, $d_{GHP}(X_N,Y_N)\leq d_\pi(M_N^X,M_N^Y)$.
It follows that, for infinitely many $N$,
\begin{equation}\label{psmp}
P(d_{GHP}(X_N,Y_N)<\epsilon)>1-\epsilon.
\end{equation}

Now, let $\delta>0$ be arbitrary and pick $N$ so that Equation~\ref{psmp} holds and $N$ is sufficiently large so that $d(X_N,X)<3\delta$ and $d(Y_N,Y)<3\delta$, with probability at least $1-6\delta$. By the triangle inequality and Equation~\ref{psmp}, it follows that 
$$P(d_{GHP}(X,Y)<\epsilon+6\delta)>1-\epsilon-6\delta.$$
As $\delta>0$ is arbitrary  and $d_{GHP}(X,Y)$ is a constant, it follows that  $d_{GHP}(\epsilon)<\epsilon$ if $\epsilon<1$.

\end{proof}

\section{Uniform bounds for sample distances}

One would expect that if metric measure spaces $X$ and $Y$ are close, then the distributions $\theta^X_N$ and $\theta^Y_N$ are close for all $N$, not just for $N$ large. We do show that this is the case. However, somewhat surprisingly, the correspondence is H\"older-$1/2$, but not H\"older-$\alpha$ for $\alpha>1/2$. In particular, the correspondence is not Lipschitz.

\begin{theorem}\label{smallN}
If $X$ and $Y$ are metric, measure spaces with $d_{GHP}(X,Y)<1/4$, then, for all $N\in\N$,
$$d_P(\theta^X_N,\theta^Y_N;d_\pi)\leq d_{GHP}(X,Y)^{1/2}.$$
\end{theorem}
\begin{proof}
Suppose $d_{GHP}(X,Y)<\epsilon$, consider embeddings from $X$ and $Y$ into a space $Z$ and a measure $\nu$ on $Z\times Z$  with marginals the pushforwards of the distributions of $X$ and $Y$ so that $\Delta(\nu)<\epsilon$. For $N$ sufficently large, let $\Omega$ be the product measure $\nu^N$ on $(Z\times Z)^N$ and consider the random variables associating to the point $\omega=(x_1,y_1,\dots,x_N,y_N)\in (Z\times Z)^N$ the matrices $M_N^X=(d(x_i,x_j))$ and $M_N^Y=(d(y_i,y_j))$. The distributions of $M_N^X$ and $M_N^Y$ are $\theta_N^X$ and $\theta_N^Y$, respectively.

For a point $\omega=(x_1,y_1,\dots,x_N,y_N)$, let $B(\omega)$ be the cardinality of the set $\{1\leq i\leq N: d(x_i,y_i)\geq \epsilon\}$. Then $B$ has a binomial distribution with parameters $N$ and $p$, with $p$ the probability that $d(x_1,y_1)>\epsilon$. Note that $p\leq \epsilon$.

\begin{lemma}\label{cases}
We have $d_\pi(M_N^X(\omega),M_N^Y(\omega))\leq \max(B(\omega)/N,2\epsilon)$.
\end{lemma}
\begin{proof}
After permutations we can assume that if $d(x_i(\omega),y_i(\omega))\geq \epsilon$, then $i\geq N-B(\omega)$.
The triangle inequality then shows that, for $1\leq i,j\leq N(1-B(\omega)/N)$,
$$\vert d(x_i(\omega),x_j(\omega))-d(y_i(\omega),y_j(\omega))\vert<2\epsilon.$$
By definition of $d_\pi$ the claim follows.  
\end{proof}

We now bound the probability that $B(\omega)/N$ is large.

\begin{lemma}
If $B$ is a $Binomial(N,p)$ random variable with $p<\epsilon<1/4$, then
$$Prob(B>N\epsilon^{1/2})<\epsilon^{1/2}.$$
\end{lemma}
\begin{proof}
As $B$ is a non-negative random variable and $E(B)=Np<N\epsilon$, the Markov inequality gives
$$P(B>N\epsilon^{1/2})\leq \frac{E(B)}{N\epsilon^{1/2}}<\frac{N\epsilon}{N\epsilon^{1/2}}=\epsilon^{1/2}.$$
\qedhere
\end{proof}

If $\epsilon<1/4$, by Lemma~\ref{cases} it follows that 
$$P(d_\pi(M_N^X(\omega),M_N^Y(\omega))>\epsilon^{1/2})<\epsilon^{1/2},$$
which implies that $d_P(\theta^X_N,\theta^Y_N;d_\pi)<\epsilon^{1/2}$. As this holds for all $\epsilon>d_{GHP}(X,Y)$, the result follows.
\end{proof}

Next, we show that the H\"older exponent $1/2$ is optimal.

\begin{theorem}\label{sharp}
Let $\alpha$ and $C>0$ be real numbers such that $1/2<\alpha<1$. Then for $\epsilon>0$ sufficiently small, there are spaces $X$ and $Y$ with $d_{GHP}(X,Y)\leq \epsilon$ so that there is an integer $N$ such that 
$$d_P(\theta^X_N,\theta^Y_N)>C\epsilon^\alpha.$$
\end{theorem}
\begin{proof}
Suppose the claimed inequality is violated for some $C>0$ and $\alpha\in (1/2,1)$. Consider spaces $X=\{a,b\}$ and $Y=\{c,d\}$ with $d(a,b)=2C$ and $d(c,d)=4C$. We consider the measures on these spaces with weights $1-\epsilon$ for $a$ and $c$ and $\epsilon$ for $b$ and $d$. Then by considering embeddings of $X$ and $Y$ into a space $Z$ with $3$ points so that the images of $a$ and $c$ under the respective embeddings coincide,  we see that $d_{GHP}(X,Y)\leq \epsilon$.


To show $d_P(\theta^X_N,\theta^Y_N)>C\epsilon^\alpha$, it suffices to show that for all joint distributions $\nu$ on $\D(N)\times \D(N)$ with marginals $\theta_N^X$ and $\theta_N^Y$, we have  the inequality
\begin{equation}\label{tpt}
P(d_\pi(M_N^X,M_N^Y)>C\epsilon^\alpha)>C\epsilon^\alpha.
\end{equation}
where the pair $(M_N^X,M_N^Y)$ has distribution $\nu$.

Assume $\epsilon<1$. Observe that each entry of a distance matrix $M_N^X$ for points in $X$ is either $2C$ or $0$ and for a distance matrix $M_N^Y$ for $Y$ is either $4C$ or $0$. Hence the difference between entries of $M_X$ and $M_Y$ is less than $C\epsilon^\alpha$ if and only if both the entries are zero. 

Pick $N$ so that we have
\begin{equation}\label{Nbound}
\frac{1}{2}<NC\epsilon^\alpha<1.
\end{equation}
If distance matrices $M_X$ and $M_Y$ for $X$ and $Y$ satisfy $d_\pi(M_X,M_Y)<C\epsilon^\alpha$, then by the above (as no rows can be omitted) they must both be the zero matrices. In particular,
\begin{equation}\label{allzero}
P(d_\pi(M_N^X,M_N^Y)>C\epsilon^\alpha)\geq P(M_N^X\neq 0).
\end{equation}

We find a lower bound for the right hand side of the above equation.
\begin{lemma}\label{matrx}
If $\epsilon>0$ is sufficiently small, $P(M_N^X\neq 0)>C\epsilon^\alpha$.
\end{lemma}
\begin{proof}
Let $x_1$, $x_2$,$\ldots$, $x_N$ be $N$ independent points sampled from $X$. 
Let $A_i$ be the event $A_i=\{x_i=b\}$. Then $M_N^X\neq 0$ is the event $\bigcup_{i=1}^N A_i\setminus \bigcap_{i=1}^N A_i$, hence
\begin{equation}\label{notconst}
P(M_N^X\neq 0)\geq P(\bigcup_{i=1}^N A_i)- P(\bigcap_{i=1}^N A_i).
\end{equation}
Now, the Bonferoni inequality gives
\begin{equation}\label{bon}
P(\bigcup_{i=1}^N A_i) \geq \sum_{i=1}^N P(A_i)-\sum_{1\leq i<j\leq n}P(A_i\cap A_j)
                       = N\epsilon- \frac{N(N-1)\epsilon^2}{2}.
\end{equation}
By Equation~\ref{Nbound}, we see that 
$$N>\frac{1}{2C\epsilon^\alpha},$$ 
so that
\begin{equation}\label{expn} 
N\epsilon>\frac{\epsilon^{1-\alpha}}{2C}.
\end{equation}
Observe that as $1/2<\alpha<1$, $1-\alpha<\alpha<1$. Hence, using Equation~\ref{expn} for $\epsilon$ sufficiently small, we have
\begin{equation}\label{maintrm}
N\epsilon- \frac{N(N-1)\epsilon^2}{2}>2C\epsilon^\alpha
\end{equation} 
Further, note that for $\epsilon$ sufficiently small, as $\alpha<1\leq N$,  
\begin{equation}\label{int}
P(\bigcap_{i=1}^N A_i)=\epsilon^N<C\epsilon^\alpha.
\end{equation}
Using Equations~\ref{notconst}, \ref{bon}, \ref{expn} and \ref{int}, we obtain that for $\epsilon$ sufficiently small.
\begin{equation}\label{lowbnd}
P(M_N^X\neq 0)\geq C\epsilon^\alpha 
\end{equation}
as required.
\end{proof}

Using Lemma~\ref{matrx} and Equation~\ref{allzero}, we obtain Equation~\ref{tpt}, completing the proof of Theorem~\ref{sharp}.

\end{proof}

\section{Relative entropy of metric measure spaces}\label{largedev}
Let $(X,d,\mu)$ be a (compact, as always) metric measure space and let $x_{1},\ldots ,x_{N}$ be points sampled independently from the measure $\mu$. Let $X_{N}=\{x_{1},\ldots ,x_{N}\}$ with the induced metric from $X$ and endowed with the measure $\mu_{N}=N^{-1}\sum_{k=1}^{N}\delta_{x_{k}}$. Then $(X_{N},d,\mu_{N})$ is a metric measure space and $d_{GHP}(X_{N},X)\rightarrow 0$ in probability  by Lemma~\ref{sampconv}. Recall that convergence in probability means  $P(d_{GHP}(X_{N},X)>\epsilon)\rightarrow 0$ for any $\epsilon>0$.  It is natural to ask  whether a {\em large deviation principle} holds for  this convergence. We quickly recall what this means. For more details we refer to the comprehensive book by Dembo and Zeitouni~\cite{dembozeitouni}.

\begin{definition} Let $\mathcal X$ denote the space of all metric measure spaces endowed with the metric $d_{GHP}$. Let $X_{N},X$ be as above. Let $I:\mathcal X \rightarrow [0,\infty]$ be a lower semicontinuous function. We say that a large deviation principle holds for the sequence $X_N$ with the  rate function $I$  if for any Borel set $A\subseteq \mathcal X$ with interior $A^{\circ}$ and closure $\bar{A}$, we have
$$
 -\inf_{Y\in A^{\circ}}I(Y)\le \liminf_{N\rightarrow \infty}\frac{\log P(X_{N}\in A)}{N}\le \limsup_{N\rightarrow \infty} \frac{\log P(X_{N}\in A)}{N}\le  -\inf_{Y\in \bar{A}}I(Y).
$$
Usually it is desirable  that the rate function be {\em good}, meaning that the set $\{Y\in \mathcal X: \ I(Y)\le t\}$ is compact in $\mathcal X$ for every $t\in [0,\infty)$.
\end{definition}
The same definition can be made for random variables taking values in  any metric space in place of $\mathcal X$ (see page 5 of \cite{dembozeitouni}). To understand the meaning of this, let $(Y,\rho,\nu)$ be any metric measure space and let $B$   be the  $\delta$-ball of radius $\delta$ in $\mathcal X$ around $Y$. If $\bar{B}$  is disjoint from $X$, then  Lemma~\ref{sampconv} implies that $P(X_{N}\in \bar{B})\rightarrow 0$. However, if the large deviation principle holds, then we will have $P(X_{N}\in \bar{B})\approx e^{-Na}$ where $a=\inf_{Z\in \bar{B}}I(Z)$. Thus the probability for the sampled metric measure space $X_{N}$ to ``look like'' a space $Y$ decays exponentially fast for $Y\not= X$ (if $a>0$). 

Apart from the naturalness of the question, another reason for asking for a large deviation principle is that if it holds, the rate function $I_{X}(Y)$ will be something that may be called the {\em relative entropy} of the metric measure space $Y$ with respect to $X$. To illustrate this point, we recall the well-known theorem of Sanov.
\begin{example}
 Let $X$ and $X_{N}$ be as above, but now we only consider the convergence of $\mu_{N}$ to $\mu$ in the space of probability measures on $X$. By Prohorov's theorem, the space of probability measures on $X$ is a compact metric space. Sanov's theorem (see page 263 of \cite{dembozeitouni}) asserts that $\mu_{N}$ satisfy a large deviation principle with the good rate function 
$$
I_{\mu}(\nu)=\begin{cases}
\int \log \frac{d\nu}{d\mu} \ d\nu & \mbox{ if }\nu \mbox{ is absolutely continuous to }\mu. \\
\infty & \mbox{ otherwise. } \end{cases}
$$
$I_{\mu}(\nu)$, often denoted $D(\nu |\!| \mu)$, is called the {\em relative entropy} or the {\em Kullback-Liebler divergence} of $\nu$ with respect to $\mu$. For example, if $\mu=\frac{1}{2}\delta_{0}+\frac{1}{2}\delta_{1}$ and $\nu=p\delta_{0}+(1-p)\delta_{1}$, then $D(\nu |\!| \mu) =\log 2 + p\log p+(1-p)\log (1-p) $.
\end{example} 
Our question is to find the corresponding rate function for the convergence of $X_{N}$ to $X$. A detailed proof would be long and technical, but we just state what we expect to be the rate function and give a heuristic reason.

\begin{definition} For two metric measure spaces $(X,d,\mu)$ and $(Y,\rho,\nu)$, define the ``relative entropy''  of $Y$ with respect to $X$ by
$$
I_{X}(Y) = \inf \left\{ D(\iota_{*}\nu|\!| \mu) : \ \iota \mbox{ is  an isometric embedding of $Y$ into $X$}\right\}.
$$
\end{definition}
Then, we expect that the large deviation principle holds for the convergence of $X_{N}$ to $X$ with the rate function $I_{X}(\cdot)$. The reason is simply as follows.

Consider any embedding $\iota:Y\to X$. The probability that the empirical measure $\mu_{N}$ falls in  $B_{\delta}(\iota_{*}\nu)$ (the ball of radius $\delta$ around $\mu$ in the space of probability measures on $X$) with a probability of about $\exp\{-N D(\iota_{*}\nu |\!| \mu)\}$, by Sanov's theorem cited earlier. Since we can choose any embedding, the one with the maximum probability should dominate, giving rise to the definition of $I_{X}(Y)$. 

 It is not our intention to give a detailed proof of this statement here, but to introduce at the quantity $I_{X}(Y)$,  which may be called the relative entropy of the metric measure space $Y$ with respect to $X$.




\begin{thebibliography}{10}
\bibitem{dembozeitouni} Dembo, A. and Zeitouni, O. \textit{Large deviations techniques and applications.} 
Corrected reprint of the second (1998) edition. Stochastic Modelling and Applied Probability, \textbf{38}  Springer-Verlag, Berlin, 2010.

\bibitem{Gr1} Gromov, M. 
\textit{Metric structures for Riemannian and non-Riemannian spaces.} Progress in Mathematics, \textbf{152} Birkhäuser Boston, Inc., Boston, MA, 1999.

\bibitem{Gr2} Gromov, M.
\textit{CAT($\kappa$)-spaces: construction and concentration.} 
J. Math. Sci. (N. Y.) \textbf{119} (2004), no. 2, 178–-200.

\bibitem{Ve} Vershik, A. M.
\textit{Random metric spaces and universality.}   
Russian Math. Surveys \textbf{59} (2004), no. 2, 259–-295. 

\end{thebibliography}
\end{document}